\documentclass[reqno,a4paper, 11pt]{amsart}

\usepackage[a4paper=true,pdfpagelabels]{hyperref}
\usepackage{graphicx}

\usepackage[ansinew]{inputenc}
\usepackage{amsfonts,epsfig}
\usepackage{latexsym}
\usepackage{mathabx}
\usepackage{amsmath}
\usepackage{amssymb}

\newcommand{\D}{\mathbb{D}}
\newcommand{\DD}{\widehat{\mathcal{D}}}
\newcommand{\Dd}{\widecheck{\mathcal{D}}}
\newcommand{\DDD}{\mathcal{D}}

\newcommand{\N}{\mathbb{N}}

\renewcommand{\H}{\mathcal{H}}
\newcommand{\B}{\mathcal{B}}

\def\tom{\widetilde{\omega}}

\def\bpnu{B_p(\nu)}

\def\binfnu{\mathcal{B}_\infty(\nu)}
\def\binfhat{\mathcal{B}_\infty(\widehat{\mathcal{D}})}
\def\binfD{\mathcal{B}_\infty(\mathcal{D})}

\newcommand{\e}{\varepsilon}
\newcommand{\vp}{\varphi}
\newcommand{\z}{\zeta}

\newcommand{\om}{\omega}

\def\g{\gamma}

\def\t{\theta}
\def\a{\alpha}
\def\b{\beta}

\newtheorem{theorem}{Theorem}
\newtheorem{lemma}[theorem]{Lemma}
\newtheorem{proposition}[theorem]{Proposition}

\newtheorem{lettertheorem}{Theorem}

\newtheorem{letterlemma}[lettertheorem]{Lemma}
\newtheorem{letterproposition}[lettertheorem]{Proposition}

\theoremstyle{definition}

\newenvironment{Prf}{\noindent{\emph{Proof of}}}
{\hfill$\Box$ }

\newenvironment{proofof}[1]
{\noindent{\it Proof of #1 }}{\hfill $\Box$\par\vspace{2.5mm}}

\numberwithin{equation}{section}
\begin{document}
\title[Weighted norm inequalities for derivatives on Bergman spaces]
{Weighted norm inequalities for derivatives on Bergman spaces}

\author{Jos\'e \'Angel Pel\'aez}
\address{Departamento de An\'alisis Matem\'atico, Universidad de M\'alaga, Campus de
Teatinos, 29071 M\'alaga, Spain. Phone number: 0034952131911} \email{japelaez@uma.es} 

\author{Jouni R\"atty\"a}
\address{University of Eastern Finland, P.O.Box 111, 80101 Joensuu, Finland}
\email{jouni.rattya@uef.fi}

\thanks{This research was supported in part by Ministerio de Ciencia Innovación y universidades, Spain, projects
PGC2018-096166-B-100 and MTM2017-90584-REDT; La Junta de Andaluc{\'i}a,
project FQM210; 
Academy of Finland 286877.}

\date{\today}

%\subjclass[2010]{30H20, 47G10}


\begin{abstract}
An equivalent norm in the weighted Bergman space $A^p_\om$, induced by an $\om$ in a certain large class of non-radial weights, is established in terms of higher order derivatives. Other Littlewood-Paley inequalities are also considered. On the way to the proofs, we characterize the $q$-Carleson measures for the weighted Bergman space $A^p_\om$ and the boundedness of a H\"ormander-type maximal function. Results obtained are further applied to describe the resolvent set of the integral operators $T_g(f)(z)=\int_0^z g'(\zeta)f(\zeta)\,d\z$ acting on $A^p_\om$.
\end{abstract}

\maketitle

\section{Introduction and main results}

Let $\H(\D)$ denote the algebra of all analytic functions in the unit disc $\D=\{z:|z|<1\}$ of the complex plane
$\mathbb{C}$. A function $\omega:\D\to [0,\infty)$, integrable over $\D$, is called a weight. It is radial if $\omega(z)=\omega(|z|)$ for all $z\in\D$ and $\int_0^1\om(s)\,ds<\infty$. For $0<p<\infty$ and a weight $\omega$, the weighted Bergman
space $A^p_\omega$ consists of those $f\in\H(\D)$ for which
    $$
    \|f\|_{A^p_\omega}^p=\int_\D|f(z)|^p\omega(z)\,dA(z)<\infty,\index{$\Vert\cdot\Vert_{A^p_\omega}$}
    $$
where $dA(z)=\frac{dx\,dy}{\pi}$ is the normalized Lebesgue area measure on $\D$.
 
In this paper we are interested in obtaining, for a large class of non-radial weights $\omega$, equivalent norms of $f$ in $A^p_\omega$ in terms of its higher order derivatives. This is a question that have been extensively studied for different classes of 
radial weights but it is not well-understood for general weights. See \cite{AlCo,APR2,BWZ} for recent developments on the topic.
These  norms are extremely valuable within the theory of concrete operators acting on these spaces. To name a few instances, they are used; in the study of Volterra type operators because they allow to get rid of the integral and they arise in a natural way in the description of its spectrum \cite{AlCo,APR2,AS,PelRat}, 
 in order to get crucial estimates in the description of Schatten classes of Toeplitz operators  \cite[p.356]{Lu87}, 
 in the boundedness of the Hilbert matrix  \cite[Proof of Theorem~2]{PelRathg} or in obtaining  $A^p_\omega$ norms of  Bergman reproducing kernels induced by radial weights   
\cite[Proof of Theorem~1]{PelRatproj}.

 \medskip
 A well-known formula ensures that for each $k\in\N$ and $0<p<\infty$ we have
	\begin{equation}\label{LP1}
	\|f\|_{A^p_\om}^p\asymp\int_\D|f^{(k)}(z)|^p(1-|z|)^{kp}\om(z)\,dA(z)+\sum_{j=0}^{k-1}|f^{(j)}(0)|^p,\quad f\in\H(\D),
	\end{equation}
if $\omega$ is a standard radial weight, that is, $\om(z)=(\alpha+1)(1-|z|^2)^\alpha$ for some $-1<\alpha<\infty$. Generalizations of this result for different classes of radial weights have been obtained in \cite{AS,PelaezRattya2019,PavP,Si}. In particular, it was recently proved~\cite[Theorem~5]{PelaezRattya2019} that \eqref{LP1} holds for a radial weight $\om$ if and only if $\om\in\DDD=\DD\cap\Dd$. Recall that a radial weight $\nu$ belongs to~$\DD$ if there exists a constant $C=C(\nu)>1$ such that the tail integral $\widehat{\nu}(z) = \int_{|z|}^1 \nu(s)\,ds$ satisfies the doubling condition $\widehat{\nu}(r)\le C\widehat{\nu}(\frac{1+r}{2})$ for all $0\le r<1$. Further, a radial weight $\nu$ belongs to~$\Dd$ if there exist constants $K=K(\nu)>1$ and $C=C(\nu)>1$ such that $\widehat{\nu}(r)\ge C\widehat{\nu}(1-\frac{1-r}{K})$ for all $0\le r<1$.

For a given $a\in\D\setminus\{0\}$, consider the interval $I_a=\left\{e^{i\theta}:|\arg(a e^{-i\theta})|\le \frac{(1-|a|)}{2}\right\}$, and let $S(a)=\{z\in\D: |z|\ge |a|,\, e^{it}\in I_a\}$ denote the Carleson square induced by $a$. We assume throughout the paper that $\om(S(a))>0$ for all $a\in\D\setminus\{0\}$. If this is not the case and $\om$ is radial, then $A^p_\om=\H(\D)$. For a weight $\nu$, $\om$ is a $\nu$-weight if $\om\nu$ is integrable. If $1<p<\infty$, a $\nu$-weight $\om$ belongs to the class $B_p(\nu)$ if there exists a constant $C=C(p,\nu,\om)$ such that
	\begin{equation}\label{bpdef}
	\sup_{S}\frac{\left(\int_S \om(z)\nu(z)\,dA(z) \right)^{\frac{1}{p}}\left(\int_S 
	\om^{-\frac{p'}{p}}(z)\nu(z)\,dA(z) \right)^{\frac{1}{p'}}}
    {\int_S \nu(z)\,dA(z)}<\infty,
	\end{equation}
where the supremum is taken over all Carleson squares $S$. We denote $B_\infty(\nu)=\bigcup_{1<p<\infty}B_p(\nu)$.
It has recently been proved    that  the class $B_p(\nu)$ describes the weights $\om$
such that  Bergman projection $P_\nu$, induced by a radial weight $\nu$, is bounded on  
$L^p_{\om\nu}$, $1<p<\infty$, whenever 
$\nu\in\DDD$ and the the Bergman reproducing kernel of $A^2_\nu$
has a particular integral representation \cite[Theorem~2]{PRW}. This result is a natural extension of a classical result  
due to Bekoll\'e and Bonami~\cite{Bek,BB} for standard weights. If $\nu(z)=(1+\eta)(1-|z|^2)^\eta$ we simply write $B_p(\eta)$ instead of $B_p((1-|z|)^\eta)$, $B_\infty(\eta)=B_\infty((1-|z|)^\eta)$ and $B_\infty= B_\infty(0)$. Nonnegative functions in the class  $B_p(\eta)$ or $B_\infty(\eta)$ are usually called the Bekoll\'e-Bonami weights. En route to describing the spectrum of the integral operator 
    \begin{displaymath}
    T_g(f)(z)=\int_{0}^{z}f(\zeta)\,g'(\zeta)\,d\zeta,\quad z\in\D, \quad g\in\H(\D),
    \end{displaymath}
on  standard Bergman spaces it was shown that \eqref{LP1} is satisfied if there exists $\eta>-1$ such that $\frac{\om}{(1-|z|)^\eta}\in B_\infty(\eta)$~\cite[Theorem~3.1]{AlCo}. The first result of this study says that the hypothesis $\frac{\om}{(1-|z|^2)^\eta}\in B_\infty(\eta)$ can be replaced by the weaker condition $\frac{\om}{\nu}\in B_\infty(\nu)$, $\nu\in\DDD$. To simplify the notation, we write $\binfnu=\bigcup_{1<p<\infty} \left\{\om: \frac{\om}{\nu}\in B_p(\nu) \right\}$, $\binfhat=\bigcup_{\nu\in\DD} \binfnu$
 and $\binfD=\bigcup_{\nu\in\DDD}\binfnu$.

\begin{theorem}\label{th:LPII}
Let $0<p<\infty$, $k\in\N$ and $\om\in \mathcal{B}_\infty(\DDD)$. Then \eqref{LP1} holds.
\end{theorem}

Observe that the set of radial weights in $\mathcal{B}_\infty(\DDD)$ coincides with $\DDD$, and hence \cite[Theorem~5]{PelaezRattya2019} implies that $\om\in\mathcal{B}_\infty(\DDD)$ is also a necessary condition for \eqref{LP1} to hold if $\om$ is radial. 
The class $\DDD$ also appears innately in the study of classical questions related to the   boundedness of the Bergman projection $P_\nu$ induced by a radial weight $\nu$  \cite[Theorem~3 and Theorem~12]{PelaezRattya2019}, which is a  frequently used tool in order to get 
Littlewood-Paley formulas in weighted Bergman spces \cite{AlCo,Zhu}. 
Moreover,
$\mathcal{B}_\infty(\DDD)$ is in a sense a much larger class than 
	$$
	\bigcup_{\eta>-1} \left\{\om\textrm{ is a weight}: \frac{\om}{(1-|z|)^\eta}\in B_\infty(\eta)\right\}
	$$
because it contains weights which may vanish on a set of positive Lebesgue area measure. It is also worth mentioning that \eqref{LP1} holds if and only if $\om\in B_\infty$, when $\om$ is essentially (or almost) constant in each hyperbolically bounded region~\cite[Corollary~4.4]{APR2}. This last condition ensures that the inequality
	\begin{equation}\label{LP2}
	\int_\D|f^{(k)}(z)|^p(1-|z|)^{kp}\om(z)\,dA(z)+\sum_{j=0}^{k-1}|f^{(j)}(0)|^p\lesssim \|f\|_{A^p_\om}^p ,\quad f\in\H(\D),
	\end{equation}
holds~\cite[Theorem~3.1]{AlCo}, see also~\cite[Theorem~A]{BWZ}. We prove the following result concerning this last inequality. 

\begin{theorem}\label{th:LPI}
Let $0<p<\infty$, $k\in\N$ and $\om\in\binfhat$. Then \eqref{LP2} holds.
\end{theorem}

Obviously there are weights in $\binfhat$ which are not essentially constant in each hyperbolically bounded region. Moreover, $\binfhat$ describes the radial weights such that \eqref{LP2} holds by \cite[Theorem~6]{PelaezRattya2019}, because the radial weights in the class $\mathcal{B}_\infty(\DD)$ coincide with $\DD$.
 
The proofs of Theorems~\ref{th:LPII} and~\ref{th:LPI} have three key ingredients. The first of them provides a geometric description of the $q$-Carleson measures for $A^p_\om$, provided $q\ge p$ and $\om\in\binfhat$. To state the result, for a given measure $\mu$ on $\D$, we write $\mu(E)=\int_E d\mu$ for each $\mu$-measurable set $E\subset\D$. Further, for each $\vp\in L^1_\om$, the H\"ormander-type maximal function~\cite{HormanderL67} is defined by 
   $$
    M_{\om}(\vp)(z)=\sup_{z\in S}\frac{1}{\om\left(S\right)}\int_{S}|\vp(\z)|\om(\z)\,dA(\z),\quad
    z\in\D.
    $$
Our characterization of $q$-Carleson measures for $A^p_\om$ reads as follows.

\begin{theorem}\label{th:cm}
Let $0<p\le q<\infty$, $\om\in\binfhat$ and
% $\nu\in\DD$ and $\omega$ a weight such that $\frac{\om}{\nu} \in B_\infty(\nu)$. 
 let $\mu$ be a positive Borel measure on $\D$. Then the following statements are equivalent:
\begin{itemize}
\item[\rm(i)] $\mu$ is a $q$-Carleson measure for $A^p_\omega$;
\item[\rm(ii)] $[M_{\om}((\cdot)^{\frac{1}{\alpha}})]^{\alpha}:L^p_\omega\to
L^q_\mu$ is bounded for each $\alpha>\frac{1}{p}$;
\item[\rm(iii)] $\mu$ satisfies 
 \begin{equation}\label{eq:s1n}
    \sup_{S}\frac{\mu\left(S\right)}{\left(\om\left(S\right)\right)^\frac{q}p}<\infty,
    \end{equation}
    where the supremum runs over all the Carleson squares $S$ in $\D$.
\end{itemize}    
Moreover,
	\begin{equation}\label{equiva}
	\|I_d\|^q_{A^p_\om\to L^q_\mu}
	\asymp \|[M_{\om}((\cdot)^{\frac{1}{\alpha}})]^{\alpha}\|^q_{L^p_\om\to L^q_\mu}
	\asymp\sup_{S}\frac{\mu\left(S\right)}{\left(\om\left(S\right)\right)^\frac{q}p}.
	\end{equation}
\end{theorem}

Theorem~\ref{th:cm} is a natural extension of \cite[Theorem~3.3]{PelSum14} and \cite[Theorem~2.1]{PelRat} to non-radial weights.

Second, it is needed a good understanding of the class of weights involved in  Theorems~\ref{th:LPII} and~\ref{th:LPI}. In particular, 
en route to the proofs, we show that 
	$$
	\om\in \mathcal{B}_\infty(\DDD)\Rightarrow \om_{[\beta]}(z)=(1-|z|)^\beta\om(z)\in \mathcal{B}_\infty(\DDD),\quad\text{ for any $\b>0$,}
	$$
a fact which might be deceptively simple-looking. Indeed, the class $\mathcal{B}_\infty(\DD)$ does not have this property. 
See \cite[Proposition~10]{PelRatHarmonic} for the construction of a radial weight   $\om\in\mathcal{B}_\infty(\DD)$
such that $\om_{[\beta]}\notin\mathcal{B}_\infty(\DD)$ for any $\b>0$.

The third key ingredient in the proofs of Theorems~\ref{th:LPII} and~\ref{th:LPI} concerns certain more smooth weights. Namely, each weight $\om$ induces the nonnegative average function 
	$$
	\widetilde{\om}(z)=\frac{\int_{S(z)}\om(\z)\,dA(\z)}{(1-|z|)^2},\quad z\in\D\setminus\{0\}.
	$$
Which regard to this function we prove the following result.
 
\begin{theorem}\label{th:averageweightII}
Let $0<p<\infty$, $k\in\N$ and $\om\in\mathcal{B}_\infty(\DDD)$. Then
	$$
	\|f\|^p_{A^p_\om}
	\asymp\|f\|^p_{A^p_{\widetilde{\om}}}
	\asymp\int_\D|f^{(k)}(z)|^p(1-|z|)^{kp}\widetilde{\om}(z)\,dA(z)+\sum_{j=0}^{k-1}|f^{(j)}(0)|^p, \quad f\in\H(\D).
	$$
\end{theorem}
 
We emphasize that, under the hypothesis $\om\in\mathcal{B}_\infty(\DDD)$, the weights $\om$ and $\widetilde{\om}$ are not pointwise equivalent, but $\|f\|^p_{A^p_\om}\asymp\|f\|^p_{A^p_{\widetilde{\om}}}$ and  $\widetilde{\om}$ is essentially (or almost) constant in each hyperbolically bounded region. This together with the second equivalence in Theorem~\ref{th:averageweightII} and \cite[Corollary~4.4-Theorem~1.7]{APR2} implies that $\widetilde{\om}\in B_\infty$.
Therefore the study of certain type of questions on linear operators $T:\H(\D)\to\H(\D)$ on weighted Bergman spaces $A^p_\om$, with $\om\in \mathcal{B}_\infty(\DDD)$, can be reduced to the particular case $\om\in B_\infty$. We will make this statement precise in the case of some questions related to the integral operator $T_g$ induced by $g\in \H(\D)$. Indeed, Theorem~\ref{th:Tgnew} below describes the analytic symbols $g$ such that $T_g:A^p_\om\to A^q_\om$ is bounded or compact. In particular, it shows that $T_g: A^p_\om\to A^p_\om$ is bounded if and only if $g$ belongs to the classical space $\B$ of Bloch functions. Further, by using ideas from studies~\cite{AlCo,AlPe,APR2}, which link the resolvent set $\rho\left(T_g|A^p_\om\right)$ to the theory of weighted norms in terms of derivatives, we obtain the following characterization of $\rho\left(T_g|A^p_\om\right)$. 

\begin{theorem}\label{th:resolvent}
Let $\om\in\mathcal{B}_\infty(\DDD)$, $g\in\B$, $0<p<\infty$ and $\lambda\in\mathbb{C}\setminus\{0\}$. Then the following statements are equivalent:
	\begin{enumerate}
	\item[(i)] $\lambda\in\rho\left(T_g|A^p_\om\right)$; 
	\item[(ii)] $\displaystyle\|f\|_{A^p_{\om_{\lambda,g,p}}}^p\asymp |f(0)|^p+\int_\D|f'(z)|^p(1-|z|)^p\om_{\lambda,g,p}(z)\,dA(z)$ for all $f\in\H(\D)$, where $\om_{\lambda,g,p}=\om\exp\left( p \text{Re}\frac{g}{\lambda} \right)$;
	\item[(iii)] $\widetilde{\om}\exp\left(p\text{Re}\frac{g}{\lambda}\right)\in B_\infty$.
\end{enumerate}    
\end{theorem}

The remaining part of the paper is organized as follows. In Section~\ref{sec:weights} we state and prove some preliminary results on weights.
Theorem~\ref{th:cm} is  proved in Section~\ref{sec:Carleson} while Section~\ref{sec:LP} is devoted to the proofs of Theorems~\ref{th:LPII},~\ref{th:LPI} and~\ref{th:averageweightII}. In Section~\ref{sec:integral} we discuss some basic properties of the integral operator $T_g$ acting on $A^p_\om$ and then prove Theorem~\ref{th:resolvent}.

Before proceeding further, a word about notation used. The letter $C=C(\cdot)$ will denote an absolute constant whose value depends on the parameters indicated in the parenthesis, and may change from one occurrence to another.
We will use the notation $a\lesssim b$ if there exists a constant
$C=C(\cdot)>0$ such that $a\le Cb$, and $a\gtrsim b$ is understood
in an analogous manner. In particular, if $a\lesssim b$ and
$a\gtrsim b$, then we write $a\asymp b$ and say that $a$ and $b$ are comparable. This notation has already been used above in the introduction.
 
\section{Basic properties and lemmas on weights}\label{sec:weights}

The pseudohyperbolic distance between two points $z_1$ and $z_2$ in $\D$ is $\rho(z_1,z_2)=\left|\frac{z_1-z_2}{1-\overline{z_1}z_2}\right|$. We say that a weight $\om$ is essentially constant on each hyperbolically bounded region if there exist constants $r\in(0,1)$ and $C=C(\om,r)$ such that
	$$
	C^{-1}\om(z_2)\le \om(z_1)\le C\om(z_2),\quad \rho(z_1,z_2)<r.
	$$ 
This class of weights coincides with the weights satisfying \cite[(1.6)]{APR2} and has been also considered in~\cite{BWZ}.

In the classical setting, there are many  equivalent conditions which describe
the Muckenhoupt class $A_\infty=\bigcup_{1\le q<\infty} A_q$, see \cite{DMOMathZ16}, \cite[Chapter~5]{SteinHarmonic93} or
\cite[p.~149]{Duandilibro}. However, this is no longer true for the class $B_\infty$,
 that is, the corresponding conditions (defined on Carleson squares) do not coincide. This stems from the fact that 
 %self-improvement of $B_p$ class to a $B_{p-\epsilon}$,\, $\epsilon>0$
%is false plays an essential role. 
$B_\infty$-weights do not have the reverse H\"older property \cite{APR2,BorichevMathAnn} .
%However, $B_\infty$-weights which are are essentially constant in each hiperbocally bounded region
%can be described in terms of a good number of these conditions \cite[Theorem~1.7]{APR2}. 
It is worth mentioning
that our definition of the class $B_\infty$ differs from the one provided in \cite[(1.4)]{APR2}. However,
this does not cause any trouble because $B_\infty$-weights, which are are essentially constant in each hyperbolically bounded region,
can be described in terms of a good number of conditions \cite[Theorem~1.7]{APR2}, and in particular the definitions
coincide under this extra hypotheses on the weight.

The main results of this paper are established under the hypothesis $\om\in\mathcal{B}(\DDD)$. Therefore we are interested in looking for neat conditions describing the class $\mathcal{B}_\infty(\nu)$ induced by $\nu\in\DDD$. In order to do this, for each weight $\nu$, we say that a weight $\om$ has the Kerman-Torchinsky $KT(\nu)$-property if there exist constants $\delta\in (0,1)$ and $C>0$ such that 
	\begin{equation}\label{eq:KT}
	\frac{\nu(E)}{\nu(S)}\le C\left( \frac{\om(E)}{\om(S)}\right)^{\delta}
	\end{equation}
for all Carleson squares $S\subset\D$ and measurable sets $E\subset S$. Here and from now on we write $\om(E)=\int_E \om(z)\,dA(z)$. If we replace $\nu$ by the Lebesgue measure in $\mathbb{R}^n$ and Carleson squares by cubes $Q$ in \eqref{eq:KT}, we obtain a condition which describes the class $A_\infty$ of the classical Muckenhoupt weights~\cite[Theorem~3.1]{DMOMathZ16}. This condition was introduced by Kerman and Torchinsky~\cite[Proposition~1]{KT} in order to describe the Hardy-Littlewood maximal operators that are of restricted weak-type. The next result  follows from \cite[Theorem~3.1(c)]{DMOMathZ16} (which holds for general bases).

\begin{letterproposition}
Let $\nu$ be a weight. Then a weight $\om$ belongs to $\mathcal{B}_\infty(\nu)$ if and only if it has the $KT(\nu)$-property.
\end{letterproposition}

The $K$-top of a Carleson box $S(a)$ is the polar rectangle $T_K(a)=\{re^{it}:\,e^{it}\in I_a,\,
|a|\le r< 1- \frac{1-|a|}{K}\}$. In some of the auxiliary results obtained en route to the main theorems the conditions $\mathcal{B}(\DDD)$ and $\mathcal{B}(\DD)$ can be relaxed in the sense that \eqref{eq:KT} is only needed for $K$-tops or their complements $S\setminus T_K$. To be precise, we write $\om\in\DD(\D)$ if there exists $C=C(\om)>0$ such that $\om(S(a))\le C\om(S(\frac{1+|a|}{2}e^{i\arg a}))$ for all $a\in\D\setminus\{0\}$. It is easy to see that each $\om\in\DD(\D)$ satisfies $\om(S(a'))\le C(C+1)\om(S(a))$ for all $a,a'\in\D\setminus\{0\}$ with $|a'|=|a|$ and $\arg a'=\arg a\pm(1-|a|)$. Therefore $\om(S(a))\lesssim\om(S(b))$ whenever $|b|=\frac{1+|a|}{2}$ and $S(b)\subset S(a)$. It is also obvious that radial weights in $\DD(\D)$ form the class $\DD$, which plays a crucial role in the operator theory of Bergman spaces induced by radial weights 
\cite{PelaezRattya2019}.	Further, a weight $\om$ on $\D$ belongs to $\Dd(\D)$ if there exist $K=K(\om)>1$ and $C=C(\om)>1$ such that
		\begin{equation}\label{eq:defDd}
		\om(S(a))\le C\om(T_K(a)),\quad a \in\D\setminus\{0\}. 
		\end{equation}	  
It is clear that radial weights in $\Dd(\D)$ form the class $\Dd$. Finally, we write $\DDD(\D)=\DD(\D)\cap\Dd(\D)$ for short.

In view of the above we have $\mathcal{B}_\infty(\DD)\subset \DD(\D)$, $\mathcal{B}_\infty(\Dd)\subset \Dd(\D)$ and
$\mathcal{B}_\infty(\DDD)\subset \DDD(\D)$. These embeddings, which will be used repeatedly throughout the paper, can also be proved
by straightforward calculations which show that $\om\in\DD(\D)$ (resp. $\om\in\Dd(\D)$) if $\nu\in\DD(\D)$ (resp. $\nu\in\Dd(\D)$) and $\om\in\binfnu$. Therefore $\DD$ and $\DDD$ coincide with the radial weights in $\binfhat$ and $\binfD$, respectively. However, $\binfhat\subsetneq\DD$. Namely, let 
\begin{equation*}\label{eq:gammadeu}\index{$\Gamma(u)$}
    \Gamma(\zeta)=\left\{z\in \D:\,|\arg\z-\arg
    z|<\frac{1}{2}\left(1-|z|\right)\right\},\quad
    \zeta\in\partial\D,
    \end{equation*}
and consider the weight $\om=\chi_{\D\setminus\Gamma(1)}$. Then $\om\notin\binfhat$ as is seen by considering the Carleson squares $S(a)$ induced by $a\in(0,1)$. But obviously there exists a constant $C>0$ such that $\om(S_a)\ge C|S_a|$ for all $a\in\D\setminus\{0\}$, and thus $\om\in\DD(\D)$.

The proof of the following result concerning the class $\DD(\D)$ can be found in \cite[Lemma~14]{KorhonenRattya2018}.

\begin{letterlemma}\label{Lemma:test-functions-non-radial}
Let $\om$ be a weight on $\D$. Then the following statements are equivalent:
    \begin{itemize}
    \item[\rm(i)] $\om\in\DD(\D)$;
    \item[\rm(ii)] there exist $\b=\b(\om)>0$ and $C=C(\om)\ge1$ such that
        $$
        \frac{\om(S(a))}{(1-|a|)^\b}\le C\frac{\om(S(a'))}{(1-|a'|)^\b},\quad 0<|a|\le|a'|<1,\quad\arg a=\arg a';
        $$
    \item[\rm(iii)] for some (equivalently for each) $K>0$ there exists $C=C(\om,K)>0$ such that
        $$
        \om(S(a))\le C\om\left(S\left(\frac{K+|a|}{K+1}e^{i\arg a}\right)\right),\quad a\in\D\setminus\{0\};
        $$
    \item[\rm(iv)] there exist $\eta=\eta(\om)>0$ and $C=C(\eta,\om)>0$ such that
        $$
        \displaystyle
        \int_\D\frac{\om(z)}{|1-\overline{a}z|^\eta}\,dA(z)\le C\frac{\om(S(a))}{(1-|a|)^\eta},\quad a\in\D\setminus\{0\}.
        $$
    \end{itemize}
\end{letterlemma}

The following lemma gives an analogue of Lemma~\ref{Lemma:test-functions-non-radial}(ii) for weights in $\Dd(\D)$.

\begin{lemma}\label{le:Dd}
Let $\om$ be a weight on $\D$.
% such that $\om(S(a))>0$ for all $a\in\D\setminus\{0\}$. 
Then $\om\in\Dd(\D)$ if and only if there exist $K=K(\om)>1$ and $\beta_0=\beta_0(\om)>0$ such that
	\begin{equation}\label{Eq:d-check-characterization}
	\om(S(a))\ge\left(\frac{1-|a|}{1-|b|}\right)^\b\om\left(S(a)\setminus D(0,|b|)\right),\quad 1-\frac{1-|a|}{K}\le|b|<1,
	\end{equation}
for all $0<\b\le\b_0$ and $a\in\D\setminus\{0\}$.
\end{lemma}

\begin{proof}
First observe that $\om\in\Dd(\D)$ if and only if there exist $C=C(\om)>1$ and $K=K(\om)>1$ such that 
	\begin{equation}\label{Eq:d-check-characterization-trivial}
	\om(S(a))\ge C\om\left(S(a)\setminus T_K(a)\right),\quad a\in\D\setminus\{0\}.
	\end{equation}
This is the characterization that we will use to prove the lemma. 

The choice $b=1-\frac{1-|a|}{K}$ in \eqref{Eq:d-check-characterization} implies \eqref{Eq:d-check-characterization-trivial} with $C=K^\b$, and therefore $\om\in\Dd(\D)$. To prove the converse implication, assume without loss of generality that $K\in\N$. Now divide $S(a)\setminus T_K(a)$ into $K$ Carleson squares of equal size and apply \eqref{Eq:d-check-characterization-trivial} to each square to obtain $\om(S(a))\ge C^2\om\left(S(a)\setminus T_{K^2}(a)\right)$ for all $a\in\D\setminus\{0\}$. Then divide $S(a)\setminus T_{K^2}(a)$ into $K^2$ squares and proceed. After $1+K+K^2+\cdots+K^{n-1}$ applications of \eqref{Eq:d-check-characterization-trivial} we obtain
	\begin{equation}\label{himppu}
	\om(S(a))\ge C^n\om\left(S(a)\setminus T_{K^n}(a)\right),\quad a\in\D\setminus\{0\}.
	\end{equation}
Now, for given $1-\frac{1-|a|}{K}\le|b|<1$, pick up $n=n(a,b)\in\N$ such that 
	$$
	1-\frac{1-|a|}{K^n}\le|b|<1-\frac{1-|a|}{K^{n+1}}\quad\Longleftrightarrow\quad K^n\le\frac{1-|a|}{1-|b|}<K^{n+1}.
	$$
Then \eqref{himppu} yields
	\begin{equation*}
	\begin{split}
	\om(S(a))
	&\ge K^{n\log_K C}\om\left(S(a)\setminus T_{K^n}(a)\right)
	>\left(\frac{1-|a|}{1-|b|}\right)^{\frac{n}{n+1}\log_KC}\om\left(S(a)\setminus T_{\frac{1-|a|}{1-|b|}}(a)\right)
    \\ & \ge \left(\frac{1-|a|}{1-|b|}\right)^{
   \frac{1}{2}\log_KC }\om\left(S(a)\setminus D(0,|b|)\right),
	\end{split}
	\end{equation*}
which gives \eqref{Eq:d-check-characterization} for $\beta_0=\frac{1}{2}\log_KC$.
\end{proof}
 
For any $\epsilon\in (0,1)$, a simple computation shows that the weight	
		\begin{equation*}
  W(re^{i\t})=\left\{
        \begin{array}{cl}
        \frac{1}{(1-r)^{1-\frac{\epsilon}{2}}|\theta|^{1-{\frac{\epsilon}{2}}}},&\quad \theta\neq 0\\
        1,&\quad\theta=0,
        \end{array}\right.
  \end{equation*}
is a Bekoll\'e-Bonami type weight such that $W(S(a))\asymp(1-a)^{\epsilon}$ if $a\in(0,1)$ which in particular implies that $H^p\nsubset A^p_W$ for any $0<p<\infty$ by the classical Carleson embedding theorem. Now, let us compare this example with Theorem~\ref{th:averageweightII}, which in particular says that $\tilde{\om}$ is a weight whenever $\om\in \binfD$.
This example says that despite the fact $\tilde{\om}$ is a weight if $\om\in\binfD$, given $\e>0$, there are $\om\in\binfD$ and a set $A\subset \D$ with $|A|=0$ 
and $\bar{A}\cup\partial\D\neq \emptyset$,
 such that $\om(S(a))\asymp (1-|a)^\epsilon$, as  $a\in A$ and $|a|\to 1^-$.

\section{Carleson measures}\label{sec:Carleson}

Let $X$ be a quasi-Banach space of analytic functions on~$\D$. A positive Borel measure~$\mu$ on~$\D$ is called a $q$-Carleson measure for~$X$
if the identity operator $I_d:X\to L^q_\mu$ is bounded. Moreover, if $I_d:X\to L^q_\mu$ is compact, then $\mu$ is a $q$-vanishing Carleson measure for $X$.

We begin with the boundedness of the H\"ormander-type maximal function on $L^p_\om$ when $\om\in\DD(\D)$.

\begin{proposition}\label{co:maxbou}
Let $0<p\le q<\infty$ and $0<\alpha<\infty$ such that $p\alpha>1$. Let $\om\in\DD(\D)$ and let $\mu$ be a positive Borel measure on~$\D$.  Then $[M_{\om}((\cdot)^{\frac{1}{\alpha}})]^{\alpha}:L^p_\omega\to
L^q_\mu$ is bounded if and only if $\mu$ satisfies \eqref{eq:s1n}.
Moreover,
    $$
    \|[M_{\om}((\cdot)^{\frac{1}{\alpha}})]^{\alpha}\|^q_{L^p_\om\to L^q_\mu}
		\asymp\sup_{S}\frac{\mu\left(S\right)}{\left(\om\left(S\right)\right)^\frac{q}p}.
    $$
\end{proposition}

Proposition~\ref{co:maxbou} can be established by following the lines of the proof of \cite[Theorem~3]{PelRatEmb}. We omit the details of the argument. A similar result was obtained in~\cite[Theorem~1.1]{KFCol16} under stronger hypotheses on $\om$.

%\begin{proposition}\label{pr:cm}
%Let $0<p,q<\infty$ and $\omega\in\DD(\D)$, and let $\mu$ be a positive Borel measure on $\D$. If $\mu$ is a $q$-Carleson measure for $A^p_\omega$, then
% \begin{equation}\label{eq:s1}
%    \sup_{S}\frac{\mu\left(S \right)}{\left(\om\left(S
%    \right)\right)^\frac{q}p}\lesssim \| I_d\|^q_{A^p_\om\to L^q_\mu}<\infty.
%    \end{equation}
%\end{proposition}

%\begin{proof} 
%\end{proof}

\begin{proofof}{\em{  Theorem~\ref{th:cm}. }}
We will show first that (i) implies (iii) under the weaker hypothesis $0<p,q<\infty$ and $\omega\in\DD(\D)$. To see this, for $a\in\D$ and $0<p,\gamma<\infty$, consider the test functions
	\begin{equation}\label{eq:Fa}
	F_{a,p,\g}(z)=\left(\frac{1-|a|^2}{1-\overline{a}z}\right)^{\frac{\gamma}{p}},\quad z\in\D.
	\end{equation}
Pick up $\gamma=\gamma(p,\om)>0$ sufficiently large such that $\frac{\gamma}{p}>\eta$, where $\eta=\eta(\om)>0$ is that of
Lemma~\ref{Lemma:test-functions-non-radial}(iv). Then
    \begin{equation*}
    \begin{split}
    \mu(S(a))&\lesssim\int_{S(a)}|F_{a,p}(z)|^q\,d\mu(z)\le\int_{\D}|F_{a,p}(z)|^q\,d\mu(z)\lesssim\|F_{a,p}\|_{A^p_\om}^q\lesssim\om\left(S(a)\right)^\frac{q}{p},\quad a\in\D,
    \end{split}
    \end{equation*}
and thus $\mu$ satisfies (iii).

The statements (ii) and (iii) are equivalent by Proposition~\ref{co:maxbou}. Hence, to complete the proof, it suffices to show that (iii) implies $\mu$ is a $q$-Carleson measure for $A^p_\om$. Since $\frac{\om}{\nu}\in B_{p_0}(\nu)$ for some $p_0>1$ and $\nu\in\DD$ by the hypothesis $\om\in\binfhat$, for any Carleson square $S$ and any non-negative $\varphi \in L^{p_0}(\om)$, H\"older's inequality yields
	\begin{equation*}
	\begin{split}
	\frac{1}{\nu(S)}\int_S \varphi \nu\,dA &\le  \frac{1}{\nu(S)} \left(\int_S \varphi^{p_0} \om \,dA\right)^{\frac{1}{p_0}} 
	\left( \int_S\left( \frac{\nu}{\om^{\frac{1}{p_0}}}\right)^{p'_0} \,dA\right)^{\frac{1}{p'_0}} 
	\lesssim\left(\frac{1}{\om(S)}\int_S \varphi^{p_0} \om \,dA\right)^{\frac{1}{p_0}}.
	\end{split}
	\end{equation*}
It follows that $\left(M_\nu(\varphi)\right)^{p_0}\lesssim M_\om(\varphi^{p_0})$ on $\D$. This together with \cite[Lemma~3.2]{PelSum14} shows that for each $s>0$ there exists a constant $C=C(s,\om)>0$ such that 
	\begin{equation}\label{eq:maxdooubling}
	|f(z)|^s\le C M_\om(f^s)(z),\quad z\in\D,\quad f\in \H(\D).
	\end{equation}
By choosing $s=\frac{1}{\alpha}<p$, and using the equivalence between (ii) and (iii) we deduce
	\begin{equation*}
	\begin{split}
	\|f\|^q_{L^q_\mu} 
	&\lesssim\int_\D \left( M_\om(f^{\frac{1}{\alpha}})(z) \right)^{q\alpha}\,d\mu(z)
	\le\|[M_{\om}((\cdot)^{\frac{1}{\alpha}})]^{\alpha}\|^q_{L^p_\om\to L^q_\mu}\| f\|^q_{A^p_\om}.
	\end{split}
	\end{equation*}
To finish the proof of the theorem we observe that \eqref{equiva} follows from the arguments above.
\end{proofof}

For the sake of completeness we describe the $q$-vanishing Carleson measures for $A^p_\om$.

\begin{theorem}\label{th:cmcompact}
Let $0<p\le q<\infty$ and $\om\in\binfhat$, and let $\mu$ be a positive Borel measure on~$\D$. Then $I_d:A^p_\om\to L^q_\mu$ is compact if and only if 
	\begin{equation}\label{eq:compact}
	\lim_{|S|\to0}\frac{\mu\left(S\right)}{\left(\om\left(S\right)\right)^\frac{q}p}=0.
	\end{equation}
\end{theorem}

\begin{proofof}{\em{Theorem~\ref{th:cmcompact}}.}
Let $0<p\le q<\infty$ and $\om\in\binfhat$, and first that assume that $I_d:\, A^p_\om\to L^q_\mu$ is compact. For each
$a\in\D$, consider the function
    \begin{equation}\label{testfunctions}
    f_{a,p,\gamma}(z)= F_{a,p,\gamma}(z) \om\left(S(a)\right)^{-\frac1p},\quad z\in\D,
    \end{equation}
where $ F_{a,p,\gamma}$ is the function defined in \eqref{eq:Fa}. Then by repeating the argument of \cite[Theorem~2.1(ii)]{PelRat} and using Lemma~\ref{Lemma:test-functions-non-radial}, we deduce 
    $$
    \lim_{|a|\to1^-}\frac{\mu\left(S(a) \right)}{\left(\om\left(S(a)
    \right)\right)^\frac{q}p}=0,
    $$
and thus \eqref{eq:compact} is satisfied.
		
Conversely, assume that $\mu$ satisfies \eqref{eq:compact}, and
set
    $$
    d\mu_r(z)=\chi_{\left\{r\le |z|<1\right\}}(z)\,d\mu(z),\quad z\in\D.
    $$
Then Theorem~\ref{th:cm} implies
    $$
    \|h\|_{L^q_{\mu_r}}\lesssim K_{\mu_r}\|h\|_{A^p_\om},\quad h\in A^p_\om,
    $$
where $K_{\mu_r}=\sup_{a\in\D\setminus\{0\}}\frac{\mu_r\left(S(a) \right)}{\left(\om\left(S(a)\right)\right)^\frac{q}p}$. We will prove next that     
		\begin{equation}\label{kmur}
    \lim_{r\to 1^-}K_{\mu_r}=0,
    \end{equation}
and then the rest of the proof follows as that of \cite[Theorem~2.1(ii)]{PelRat}. By the assumption, for a given $\e>0$ , there exists $r_0\in(0,1)$ such that
	\begin{equation}\label{eq:c1}
	\sup_{a\in\D: \,|a|\ge r_0}\frac{\mu\left(S(a) \right)}{\left(\om\left(S(a)\right)\right)^\frac{q}p}<\e.
	\end{equation}
Therefore for each $r\in(0,1)$, we have 
	\begin{equation}\label{eq:c2}
	\sup_{a\in\D: \,|a|\ge r_0}\frac{\mu_r\left(S(a) \right)}{\left(\om\left(S(a)\right)\right)^\frac{q}p}
	\le \sup_{a\in\D: \,|a|\ge r_0}\frac{\mu\left(S(a) \right)}{\left(\om\left(S(a)\right)\right)^\frac{q}p}<\e.
  \end{equation}
Next, if $|a|<r_0$, we choose $n\in\N\setminus\{1\}$ such that $(n-1)(1-r_0)<|I_a|\le n(1-r_0)$. Let $I_k$ be arcs on the boundary such that $|I_k|=1-r_0$ for all $k=1,\dots n$, and $I_a\subset\cup_{k=1}^n I_k\subset 2I_a$, where $2I_a=\left\{e^{i\theta}:|\arg(a e^{-i\theta})|\le (1-|a|)\right\}$,
where $I_j$ and $I_m$, $j\neq m$, $j,m\in\{1,2,\dots n\}$ are disjoint or share an endpoint. Let $r\ge r_0$. Then, since $\om\in\DD(\D)$ by the hypothesis, \eqref{eq:c1} yields
\begin{equation*}\begin{split}
\mu_r\left( S(a) \right) &\le \mu_{r_0}\left( S(a) \right)\le \sum_{k=1}^n \mu\left( S(I_k) \right)\le \e \sum_{k=1}^n
\left(\om\left(S(I_k)
    \right)\right)^\frac{q}p
    \\ & \le \e \left( \sum_{k=1}^n
\om\left( S(I_k) \right) \right)^{\frac{q}{p}}
    \le \e \om \left ( S(2I_a) \right)^{\frac{q}{p}} 
		\lesssim\e \om \left ( S(a) \right)^{\frac{q}{p}}.
	\end{split}
	\end{equation*}
This together with \eqref{eq:c2} gives \eqref{kmur}, and finishes the proof.
\end{proofof}

\section{Littlewood-Paley inequalities }\label{sec:LP}

We begin with Theorem~\ref{th:averageweightII}, splitting its proof in two parts. We first establish an equivalent norm to $\|\cdot\|_{A^p_\om}$ and a Littlewood-Paley inequality in terms of the average weight 
	$$
	\om_{h,r}(z)=\frac{\int_{\Delta(z,r)}\om(\z)\,dA(\z)}{(1-|z|)^2},\quad z\in\D,
	$$
where $r\in(0,1)$, and $\Delta(z,r)=\{u\in\D: \rho(u,z)<r\}$.

\begin{proposition}\label{pr:averageweight}
Let $\om\in\binfhat$, $0<r<1$ and $0<p<\infty$. Then the following statements hold:
	\begin{enumerate}
	\item[\rm(i)]
	$\displaystyle\|f\|_{A^p_\om}\asymp \|f\|_{A^p_{\om_{h,r}}}$ for all $f\in\H(\D)$;
	\item[\rm(ii)]
	$\displaystyle\int_\D|f^{(k)}(z)|^p(1-|z|)^{kp}\om_{h,r}(z)\,dA(z)+\sum_{j=0}^{k-1}|f^{(j)}(0)|^p
	\lesssim\|f\|^p_{A^p_\om}$ for all $f\in\H(\D)$.	
	\end{enumerate}
\end{proposition}

\begin{proof}
(i) Let $0<r<1$ be fixed. Then Fubini's theorem yields
	\begin{equation*}
	\begin{split}
	\int_{S(a)}\frac{\om(\Delta(\zeta,r))}{(1-|\z|)^2}\,dA(\zeta)
	&=\int_{\{z\in\D:S(a)\cap\Delta(z,r)\ne\emptyset\}}\left(\int_{S(a)\cap\Delta(z,r)}\frac{dA(\z)}{(1-|\z|)^2}\right)\om(z)\,dA(z)\\
	&\le\int_{S(b)}\left(\int_{\Delta(z,r)}\frac{dA(\z)}{(1-|\z|)^2}\right)\om(z)\,dA(z)\asymp\om(S(b)),\quad |a|>r,
	\end{split}
	\end{equation*}
where $b=b(a,r)\in\D$ satisfies $\arg b=\arg a$ and $1-|b|\asymp1-|a|$ for all $a\in\D\setminus\overline{D(0,r)}$. Since $\binfhat\subset\DD(\D)$, $\om(S(b))\lesssim\om(S(a))$ by Lemma~\ref{Lemma:test-functions-non-radial}(ii), and therefore Theorem~\ref{th:cm} yields
	\begin{equation}\label{eq:average}
  \|f\|^p_{A^p_{\om_{h,r}}}\lesssim\|f\|_{A^p_\om}^p,\quad f\in\H(\D).
	\end{equation}
To see the converse inequality, use the subharmonicity of $|f|^p$ and Fubini's theorem to deduce
	\begin{equation*}
	\begin{split}
	\|f\|_{A^p_\om}^p
	\lesssim\int_{\D}\om(\z)\left(\int_{\Delta(\z,r)}\frac{|f(z)|^p}{(1-|z|)^2}\,dA(z)\right)dA(\z)
	=\|f\|^p_{A^p_{\om_{h,r}}},\quad f\in\H(\D).
	\end{split}
	\end{equation*}
Thus (i) is proved.

(ii) Let $0<r<1$ be fixed. It is well known that, for each $0<p<\infty$, $k\in\N$ and $0<s<1$, we have  
    \begin{equation}\label{Eq:suharmonic-n-derivatives}
    |f^{(k)}(z)|^p\lesssim\frac{1}{(1-|z|)^{2+kp}}\int_{\Delta(z,s)}|f(\z)|^p\,dA(\z),\quad z\in\D,\quad f\in\H(\D),
    \end{equation}
see, for example, \cite[Lemma~2.1]{Luecking1985} for details. Fix now $0<s<1$ such that $R=s+r<1$. Then an application of \eqref{Eq:suharmonic-n-derivatives}, Fubini's theorem and Part (i) give
	\begin{equation*}
	\begin{split}
	\int_\D|f^{(k)}(z)|^p(1-|z|)^{kp}\om_{h,r}(z)\,dA(z) 
	&\lesssim\int_\D\left(\int_{\Delta(z,s)}\frac{|f(\z)|^p}{(1-|\z|)^2} \,dA(\z)\right)\om_{h,r}(z)\,dA(z)\\
	&=\int_\D \frac{|f(\z)|^p}{(1-|\z|)^2}\left(\int_{\Delta(\z,s)}\om_{h,r}(z)\,dA(z)\right)dA(\z)\\
	&\lesssim \|f\|^p_{A^p_{\om_{h,R}}}
	\asymp\|f\|_{A^p_\om}^p,\quad f\in\H(\D).
	\end{split}
	\end{equation*}
Moreover, for each $j\in\N$, $|f^{(j)}(0)|^p\lesssim\int_{D(0,\frac12)}|f|^p\,dA$ by the subharmonicity of $|f|^p$, and therefore Theorem~\ref{th:cm} implies $|f^{(j)}(0)|^p\lesssim\|f\|_{A^p_\om}^p$ once we show that $\int_{S\cap D(0,\frac12)}\,dA\lesssim\om(S)$ for all Carleson squares $S$. This last inequality is obviously valid if $S=S(a)$ with $|a|\ge\frac12$ because in this case the left hand side equals zero. For $|a|\le\frac12$ we have
		$$
		\int_{S(a)\cap D(0,\frac12)}\,dA
		\le\frac{1}{8}
		\le\frac{1}{8}\frac{\om(S(a))}{\inf_{a\in\overline{D(0,\frac12)}}\om(S(a))}
		\lesssim\om(S(a)).
		$$
This finishes the proof.
\end{proof}

Given a weight $\om$ and $\beta\in\mathbb{R}$, we denote $\om_{[\beta]}(z)=(1-|z|)^\beta\om(z)$ for all $z\in\D$. We will use this definition to shorten the notation in several instances in the proofs from here after.

\begin{proposition}\label{th:wide}
Let $k\in\N$. Then the following statements hold:
\begin{enumerate}
\item[(i)] If $0<p\le 1$ and $\om\in\binfhat$, then
	$$
	\|f\|^p_{A^p_\om}\lesssim\int_\D|f^{(k)}(z)|^p(1-|z|)^{kp}\widetilde{\om}(z)\,dA(z)+\sum_{j=0}^{k-1}|f^{(j)}(0)|^p , \quad f\in\H(\D).
	$$
\item[(ii)] If $1<p<\infty$ and $\om\in\mathcal{B}_\infty(\DDD)$, then
	$$
	\|f\|^p_{A^p_\om}\lesssim\int_\D|f^{(k)}(z)|^p(1-|z|)^{kp}\widetilde{\om}(z)\,dA(z)+\sum_{j=0}^{k-1}|f^{(j)}(0)|^p , \quad f\in\H(\D).
	$$
\end{enumerate}
\end{proposition}

\begin{proof}
(i).  Let $0<p\le1$. First observe that $\om\in\binfhat\subset\DD(\D)$,  
 and hence, by Lemma~\ref{Lemma:test-functions-non-radial}(ii) and Theorem~\ref{th:cm}, there exists $\b_0=\beta_0(\om,p)>0$ such that $A^p_\om$ is continuously embedded into $A^1_{\beta-1}$ for all $\b\ge\b_0$. A well-known reproducing formula for functions in $A^1_{\b-1}$ \cite[Proposition~4.27]{Zhu} now guarantees the estimate
	\begin{equation}\label{eq:estimate}
	\left|f(z)-\sum_{j=0}^{k-1}f^{(j)}(0)\right|
	\lesssim\int_\D\left|\frac{f^{(k)}(\zeta)}{(1-\overline{z}\zeta)^{1+\beta}}\right|(1-|\z|)^{\beta+k-1}\,dA(\z),\quad z\in\D.
	\end{equation}
Fix $\b\ge\b_0$ sufficiently large such that $p(1+\b)\ge\eta$, where $\eta=\eta(\om)>0$ is that of Lemma~\ref{Lemma:test-functions-non-radial}(iv), and $\alpha=p(\b+k+1)-2>-1$. Then $A^p_\alpha\subset A^1_{\b+k-1}$ by a well-known embedding that can be also deduced from Theorem~\ref{th:cm}, and hence
	$$
		\left|f(z)-\sum_{j=0}^{k-1}f^{(j)}(0)\right|^p\lesssim\int_\D\left|f^{(k)}(\zeta)\right|^p\frac{(1-|\z|)^\a}{|1-\overline{z}\zeta|^{p(1+\beta)}}\,dA(\z),\quad z\in\D.
	$$
Therefore Fubini's theorem and Lemma~\ref{Lemma:test-functions-non-radial}(iv) yield
	\begin{equation}
	\begin{split}\label{eq:aver1}
	\left\|f-\sum_{j=0}^{k-1}f^{(j)}(0)\right\|_{A^p_\om}^p
	&\lesssim\int_\D\left|f^{(k)}(\zeta)\right|^p(1-|\z|)^\a\left(\int_\D\frac{\om(z)}{|1-\overline{z}\zeta|^{p(1+\beta)}}\,dA(z)\right)dA(\z)\\
	&\lesssim\int_\D\left|f^{(k)}(\zeta)\right|^p(1-|\z|)^{kp}\frac{\om(S(\zeta))}{(1-|\z|)^2}\,dA(\z),\quad f\in\H(\D).
	\end{split}
	\end{equation}
Thus (i) is proved.

(ii) Let now $1<p<\infty$. Observe that $\om\in\mathcal{B}(\DDD)\subset\DDD(\D)$. We begin with showing that for each $\om\in\DDD(\D)$ there exists $\e_0>0$ such that $\om_{[-\e]}\in\DDD(\D)$ for all $0<\e<\e_0$. To see this first note that by Lemma~\ref{le:Dd} there exists $\b=\b(\om)>0$ such that $\om(\D\setminus D(0,r))\lesssim(1-r)^\beta$ for all $0\le r<1$. This and Fubini's theorem yield
	\begin{equation*}
	\begin{split}
	\int_{\D\setminus D(0,\frac12)}\frac{\om(z)}{(1-|z|)^\e}\, dA(z)
	&\asymp\int_{\D\setminus D(0,\frac12)}\om(z)\left(\int_{D(0,|z|)}\frac{dA(\zeta)}{(1-|\z|)^{1+\e}}\right)\,dA(z)\\
	&=\int_\D\frac1{(1-|\z|)^{1+\e}}\left(\int_{\D\setminus D\left(0,\max\left\{\frac12,|\z|\right\}\right)}\om(z)\,dA(z)\right)\,dA(\z)\\
	&\lesssim\int_\D\frac{dA(\z)}{(1-|\z|)^{1+\e-\b}}<\infty
	\end{split}
	\end{equation*}
for each $\e<\b$. Thus $\om_{[-\e]}$ is a weight all $0<\e<\b$. Further, for such an $\e$, the inequality \eqref{Eq:d-check-characterization} and the hypothesis $\om\in\Dd(\D)$ yield 
\begin{equation*}
\begin{split}
\int_{S(a)}\frac{\om(z)}{(1-|z|)^\e}\,dA(z)
&=\int_{S(a)}\om(z)\left(\e\int_0^{|a|}\frac{dt}{(1-t)^{\e+1}}+\e\int_{|a|}^{|z|}\frac{dt}{(1-t)^{\e+1}}+1\right)\,dA(z)\\
&=\frac{\om(S(a))}{(1-|a|)^\e}+\e\int_{|a|}^1\left(\int_{S(a)\setminus D(0,t)}\om(z)\,dA(z)\right)\frac{dt}{(1-t)^{\e+1}}\\
&\lesssim\frac{\om(S(a))}{(1-|a|)^\e}+
\frac{\om(S(a))}{(1-|a|)^\b}\e\int_{|a|}^1\frac{dt}{(1-t)^{\e+1-\b}}\\
&=\frac{\om(S(a))}{(1-|a|)^\e}+
\frac{\om(S(a))}{(1-|a|)^\e}\frac{\e}{\b-\e}\\
&\lesssim\frac{\om(T_K(a))}{(1-|a|)^\e}\le\int_{T_K(a)}\frac{\om(z)}{(1-|z|)^\e}\,dA(z),\quad a\in\D\setminus\{0\},
\end{split}
\end{equation*}
and thus $\om_{[-\e]}\in\Dd(\D)$, provided $\om\in\Dd(\D)$ and $0<\e<\b$. Further, since $\om\in\Dd(\D)$, there exists $K=K(\om,\e)>1$ such that $\om_{[-\e]}(S(a))\lesssim\om_{[-\e]}\left(T_K(a)\right)$ and $\om(S(a))\lesssim\om\left(T_K(a)\right)$ for all $a\in\D\setminus\{0\}$. Write $a'=\frac{1+|a|}{2}e^{i\arg a}$ for short. Then the inequalities just obtained and the hypothesis $\om\in\DD(\D)$ yield
	\begin{equation}\label{Eq:giuln}
	\begin{split}
	\om_{[-\e]}(S(a))
	&\lesssim\om_{[-\e]}\left(T_K(a)\right)
	\le\frac{K^\e\om\left(T_K(a)\right)}{(1-|a|)^\e}
	\le\frac{K^\e\om\left(S(a)\right)}{(1-|a|)^\e}
	\lesssim\frac{\om\left(S\left(a'\right)\right)}{(1-|a'|)^\e}\\
	&\lesssim\frac{\om\left(T_K\left(a'\right)\right)}{(1-|a'|)^\e}
	\le\om_{[-\e]}\left(T_K\left(a'\right)\right)
	\le\om_{[-\e]}\left(S\left(a'\right)\right),\quad a\in\D\setminus\{0\},
	\end{split}
	\end{equation}
and hence $\om_{[-\e]}\in\DD(\D)$, for all $0<\e<\b$. Therefore we have shown that $\om_{[-\e]}\in\DDD(\D)$ for all $0<\e<\e_0=\b$.

To prove the statement (ii) of the proposition, fix $\alpha_2\in\left(\frac{2}{p'},\frac{2}{p'}+\frac{\e_0}{p}\right)$, where $\e_0=\e_0(\om)>0$ is the constant we just found. Then $\e=\frac{p}{p'}(p'\alpha_2-2)\in(0,\e_0)$, and thus $\om_{[-\e]}\in\DD(\D)$. Let now $\b>\max\{\b_0,(\eta+2(p-1)+\e_0)/p-1\}$, where $\eta=\eta(\om_{[-\e]})>0$ is that of Lemma~\ref{Lemma:test-functions-non-radial}(iv) and $\b_0=\b_0(\om,1)$ is that of Part (i). Finally, write $1+\b=\alpha_1+\alpha_2$. Then the estimate \eqref{eq:estimate} and H\"older's inequality imply
	\begin{equation*}
	\begin{split}
		\left|f(z)-\sum_{j=0}^{k-1}f^{(j)}(0)\right|^p
	&\lesssim\int_\D |f^{(k)}(\zeta)|^p\frac{(1-|\z|)^{p(\beta+k-1)}}{|1-\overline{z}\z|^{p\alpha_1}}\,dA(\z)
	\left(\int_\D\frac{dA(\z)}{|1-\overline{z}\z|^{p'\alpha_2}}\right)^\frac{p}{p'}\\
	&\lesssim\int_\D  |f^{(k)}(\zeta)|^p\frac{(1-|\z|)^{p(\beta+k-1)}}{|1-\overline{z}\z|^{p\alpha_1}}\,dA(\z)(1-|z|)^{(2-p'\alpha_2)\frac{p}{p'}},
	\end{split}
	\end{equation*}
because $p'\alpha_2>2$. By using this and Fubini's theorem we deduce
	\begin{equation*}
	\begin{split}
	\left\|f-\sum_{j=0}^{k-1}f^{(j)}(0)\right\|_{A^p_\om}^p
	&\lesssim\int_\D |f^{(k)}(\zeta)|^p(1-|\z|)^{p(\beta+k-1)}
    \left(\int_\D\frac{\om(z)(1-|z|)^{(2-p'\alpha_2)\frac{p}{p'}}}{|1-\overline{z}\zeta|^{p\alpha_1}}\,dA(z)\right)dA(\z)
	\end{split}
	\end{equation*}
for all $f\in\H(\D)$. Since $\e=\frac{p}{p'}(p'\alpha_2-2)\in(0,\e_0)$ and $p\alpha_1>\eta$ by our choices, we may apply Lemma~\ref{Lemma:test-functions-non-radial}(iv) to the inner integral above. This together with \eqref{Eq:giuln} imply
	\begin{equation*}
	\begin{split}%\label{eq:aver2}
	\left\|f-\sum_{j=0}^{k-1}f^{(j)}(0)\right\|_{A^p_\om}^p
	&\lesssim\int_\D |f^{(k)}(\zeta)|^p(1-|\z|)^{p(\beta+k-1)}
  \left(\int_\D\frac{\om(z)(1-|z|)^{(2-p'\alpha_2)\frac{p}{p'}}}{|1-\overline{z}\zeta|^{p\alpha_1}}\,dA(z)\right)dA(\z)\\
	&\lesssim\int_\D\ |f^{(k)}(\zeta)|^p(1-|\z|)^{p(\beta+k-1)}\frac{\om_{[-\e]}(S(\zeta))}{(1-|\zeta|)^{p\alpha_1}}\,dA(\z)\\
	&\asymp\int_\D |f^{(k)}(\zeta)|^p (1-|\z|)^{kp}\frac{\om(S(\zeta))}{(1-|\z|)^2}\,dA(\z),\quad f\in\H(\D),
	\end{split}
	\end{equation*}
and finishes the proof of (ii).
\end{proof}

With these preparations we can deduce Theorem~\ref{th:averageweightII}. Namely, it is easy to see that for each $\om\in\DDD(\D)$, and in particular for each $\om\in\mathcal{B}_\infty(\DDD)$, there exists $r_0=r_0(\om)\in(0,1)$ such that $\om_{h,r}\asymp\widetilde{\om}$ in $\D$ for each $r\ge r_0$. Therefore Theorem~\ref{th:averageweightII} follows from Propositions~\ref{pr:averageweight} and~\ref{th:wide}.

We proceed to prove Theorems~\ref{th:LPI} and~\ref{th:LPII} in the said order.

\bigskip

\begin{Prf} Theorem~\ref{th:LPI}.
Let $0<r<1$ be fixed. The inequality \eqref{Eq:suharmonic-n-derivatives}, Fubini's theorem and Proposition~\ref{pr:averageweight}(i) yield
	\begin{equation*}
	\begin{split}
	\int_\D|f^{(k)}(z)|^p(1-|z|)^{kp}\om(z)\,dA(z)
	&\lesssim\int_\D\left(\int_{\Delta(z,r)}|f(\zeta)|^p\,dA(\z)\right)\frac{\om(z)}{(1-|z|)^{2}}\,dA(z)\\
	&\asymp\int_\D|f(\zeta)|^p\frac{\om(\Delta(\zeta,r))}{(1-|\z|)^2}\,dA(\zeta)\\
	&=\|f\|^p_{A^p_{\om_{h,r}}}\asymp\|f\|_{A^p_\om}^p,\quad f\in\H(\D).
	\end{split}
	\end{equation*}
Moreover, arguing as in the proof of Proposition~\ref{pr:averageweight}(ii) we deduce
	$$
	\sum_{j=0}^{k-1}|f^{(j)}(0)|^p\lesssim\|f\|_{A^p_\om}^p,\quad f\in\H(\D).
	$$
By combining the above estimates we get the assertion.
\end{Prf}

\bigskip

\begin{Prf} Theorem~\ref{th:LPII}.
By Theorem~\ref{th:LPI} it suffices to prove
	\begin{equation*}
	\int_\D|f^{(k)}(z)|^p(1-|z|)^{kp}\om(z)\,dA(z)+\sum_{j=0}^{k-1}|f^{(j)}(0)|^p
  \gtrsim\|f\|_{A^p_\om}^p,\quad f\in\H(\D).
	\end{equation*}
To see this, we will need to know more about the weights involved. In particular, we want to show that $\om_{[kp]}\in B_\infty(\DDD)$ for each $k\in\mathbb{N}$ and $0<p<\infty$. We will deduce this in several steps. First observe that if $\om\in\Dd(\D)$ and $\b>0$, then $\om_{[\b]}\in\Dd(\D)$. Namely, if $\om\in\Dd(\D)$ there exists $K=K(\om)>1$ such that for each $\b>0$ we have  
	$$
	\om_{[\beta]}(T_K(a))
	\ge\frac{(1-|a|)^\beta}{K^\b}\om(T_K(a))
	\gtrsim(1-|a|)^\beta\om(S(a))
	\ge\om_{[\beta]}(S(a)), \quad a \in\D\setminus\{0\},
	$$
and hence $\om_{[\b]}\in\Dd(\D)$. If in addition $\om\in\DDD(\D)$, then $\om_{[\b]}\in\DD(\D)$. To see this, let $a\in\D\setminus\{0\}$ and write $a'=\frac{1+|a|}{2}e^{i\arg a}$ for short. Since $\om\in\Dd(\D)$, there exists $K=K(\om)>1$ such that $\om(T_K(a))\asymp\om(S(a))$ for all $a\in\D\setminus\{0\}$. By using this and the hypothesis $\om\in\DD(\D)$ we deduce
	\begin{equation*}
	\begin{split}
	\om_{[\b]}(S(a))
	&\le(1-|a|)^\b\om(S(a))
	\asymp\left(1-|a|\right)^\b\om(S(a'))\\
	&\asymp\left(1-\left(1-\frac{1-|a'|}{K}\right)\right)^\b\om(T_K(a'))\\
	&\le\om_{[\b]}(S(a')),\quad a\in\D\setminus\{0\}.
	\end{split}
	\end{equation*}
and thus $\om_{[\b]}\in\DD(\D)$. Therefore we have shown that $\om_{[\b]}\in\DDD(\D)$, provided $\om\in\DDD(\D)$ and $\b>0$.

The other property we need to know is that if $\om\in\binfD$, then $\om_{[\b]}\in\binfD$ for all $\b>0$. We will use the fact we just proved to see this
and the fact that  $\frac{\om}{\nu}\in\bpnu$ if and only if there exists a constant $C=C(\om,\nu)>0$ such that
	\begin{equation}\label{carbpnu}
	\left(\frac{\int_{S}|f(z)|\nu(z)\,dA(z)}{\nu(S)}\right)^{p}
	\le C\frac{\int_{S}|f(z)|^p\om(z)\,dA(z)}{\om(S)}
	\end{equation}
for all Carleson squares $S$ and all measurable functions $f$ on $\D$.
 Next observe that if $\nu\in\DDD(\D)$, $1<p<\infty$ and $\frac{\om}{\nu}\in B_p(\nu)$, then $\frac{\om_{[\b]}}{\nu_{[\b]}}\in B_p(\nu_{[\b]})$ for all $0<\b<\infty$. Namely, if $\frac{\om}{\nu}\in B_p(\nu)$, then \eqref{carbpnu} yields
	\begin{equation*}
	\begin{split}
	\left(\frac{\int_{S(a)}|f(z)|\nu_{[\beta]}(z)\,dA(z)}{\nu_{[\b]}(S(a))}\right)^p
	&\le C\frac{\nu(S(a))^p}{\nu_{[\b]}(S(a))^p}\frac{\int_{S(a)}|f(z)|^p(1-|z|)^{p\beta}\om(z)\,dA(z)}{\om(S(a))}\\
	&\le C\frac{\nu(S(a))^p}{\nu_{[\b]}(S(a))^p}(1-|a|)^{(p-1)\beta}\frac{\int_{S(a)}|f(z)|^p\om_{[\b]}(z)\,dA(z)}{\om(T_K(a))}\\
	&\le C\frac{(1-|a|)^{p\beta}\nu(S(a))^p}{\nu_{[\b]}(S(a))^p}\frac{\int_{S(a)}|f(z)|^p\om_{[\b]}(z)\,dA(z)}{\om_{[\b]}(T_K(a))}
	\end{split}
	\end{equation*}
for all $a\in\D\setminus\{0\}$ and all measurable functions $f$ on $\D$. Since $\nu\in\DDD(\D)$ by the hypothesis, then $\om\in\DDD(\D)$, and hence $\om_{[\b]}\in\DDD(\D)$. Therefore $\om_{[\b]}(T_K(a))\asymp\om_{[\b]}(S(a))$ for all $a\in\D\setminus\{0\}$, provided $K=K(\om,\b)>1$ is sufficiently large. Moreover, for $M=M(\nu)>1$ sufficiently large, we have
	$$
	\nu_{[\b]}(S(a))
	\le(1-|a|)^\b\nu(S(a))
	\asymp(1-|a|)^\b\nu(T_M(a))
	\le M^\b\nu_{[\b]}(T_M(a))
	\le M^\b\nu_{[\b]}(S(a))
	$$
for all $a\in\D\setminus\{0\}$. It follows that
	\begin{equation*}
	\begin{split}
	\left(\frac{\int_{S(a)}|f(z)|\nu_{[\beta]}(z)\,dA(z)}{\nu_{[\b]}(S(a))}\right)^p
	\lesssim\frac{\int_{S(a)}|f(z)|^p\om_{[\b]}(z)\,dA(z)}{\om_{[\b]}(S(a))},\quad a\in\D\setminus\{0\},
	\end{split}
	\end{equation*}
for all measurable functions $f$ on $\D$. This shows that $\frac{\om_{[\b]}}{\nu_{[\b]}}\in B_p(\nu_{[\b]})$.

Finally, by the hypothesis $\om\in\binfD$, and hence there exist $1<p<\infty$ and $\nu\in\DDD$ such that $\frac{\om}{\nu}\in B_p(\nu)$. Therefore $\frac{\om_{[\b]}}{\nu_{[\b]}}\in B_p(\nu_{[\b]})$ for all $0<\b<\infty$. Moreover, $\nu_{[\b]}\in\DDD$, and hence $\om_{[\b]}\in \binfD$.

Now we can proceed to prove the statement of the theorem. Recall that we just showed that $\om_{[kp]}\in B_\infty(\DDD)$ for each $k\in\mathbb{N}$ and $0<p<\infty$ by the hypothesis $\om\in B_\infty(\DDD)$. In particular $\om_{[kp]}\in\DDD(\D)$, and hence the inequality we are after now follows from Proposition~\ref{th:wide} and Theorem~\ref{th:cm} if we show that  
	\begin{equation}\label{Eq:fukifuki}
	\int_S(1-|\z|)^{kp}\frac{\om(S(\zeta))}{(1-|\z|)^2}\,dA(\z)\lesssim\int_S(1-|\z|)^{kp}\om(\zeta)\,dA(\z),\quad S\subset\D.
	\end{equation}
Since $\om\in\Dd(\D)$, there exists $r=r(\om)\in(0,1)$ sufficiently large such that $\om(S(\zeta))\lesssim\om(\Delta(\zeta,r))$ for all $\z\in\D$. Therefore 
	\begin{equation*}
	\begin{split}
	\int_{S(a)}(1-|\z|)^{kp}\frac{\om(S(\zeta))}{(1-|\z|)^2}\,dA(\z)
	&\lesssim\int_{S(a)}(1-|\z|)^{kp}\frac{\om(\Delta(\zeta,r))}{(1-|\z|)^2}\,dA(\z)\\
	&=\int_{\{z:S(a)\cap\Delta(z,r)\ne\emptyset\}}\om(z)\left(\int_{S(a)\cap\Delta(z,r)}(1-|\z|)^{kp-2}\,dA(\z)\right)\,dA(z)\\
	&\lesssim\int_{\{z:S(a)\cap\Delta(z,r)\ne\emptyset\}}(1-|z|)^{kp}\om(z)\,dA(z)\\
	&\le\int_{S(a')}(1-|z|)^{kp}\om(z)\,dA(z),\quad a\in\D\setminus\{0\},
	\end{split}
	\end{equation*}
where $\arg a'=\arg a$ and $1-|a'|\asymp1-|a|$ for all $a\in\D\setminus\{0\}$. Moreover, since $\om_{[kp]}\in\DD(\D)$, we have $\om_{[kp]}(S(a'))\lesssim\om_{[kp]}(S(a))$ for all $a\in\D\setminus\{0\}$. This gives \eqref{Eq:fukifuki} and finishes the proof.
\end{Prf}

\section{Spectra of integration operator}\label{sec:integral}

Let $\B$ denote the classical space of Bloch functions, $\B_0$ the little Bloch space
and $D_{a}=\left\{z\in\D: |z-a|<\frac{1-|a}{2}\right\}$ for all $a\in\D$. Recall that $\widetilde{\om}$ is essentially constant in each hyperbolically bounded region if $\om\in\mathcal{B}_\infty(\DDD)$. This together with Theorem~\ref{th:averageweightII} and \cite[Corollary~4.4-Theorem~1.7]{APR2} implies that $\tom\in B_\infty$. Therefore the next result follows from \cite[Theorem~4.1]{AlCo}, Theorem~\ref{th:averageweightII} and the fact that 
	$$
	\tom(D_{a})\asymp\om(S(a)),\quad a\in\D\setminus\{0\},
	$$
provided $\om\in\mathcal{B}_\infty(\DDD)$.

\begin{theorem}\label{th:Tgnew}
Let $\om\in\mathcal{B}_\infty(\DDD)$ and $0<p,q<\infty$. Then the following statements hold:
	\begin{enumerate}
	\item[(i)] If $0<p\le q<\infty$, then
	$T_g: A^p_\om\to A^q_\om$ is bounded if and only if 
	$$
	\sup_{a\in\D} (1-|a|)|g'(a)|\left(\om(S(a)
 \right)^{\frac1q-\frac1p}<\infty.
	$$
	\item[(ii)] If $0<p\le q<\infty$, then
	$T_g: A^p_\om\to A^q_\om$ is compact if and only if 
	$$
	\lim_{|a|\to 1^-} (1-|a|)|g'(a)|\left(\om(S(a)\right)^{\frac1q-\frac1p}=0.
	$$
  \item[(iii)] If $0<q<p<\infty$, then $T_g: A^p_\om\to A^q_\om$ is bounded (equivalently compact) if and only if $g\in A^s_\om$, where $\frac1s=\frac1p-\frac1q$. 
	\end{enumerate}
\end{theorem}

Theorem~\ref{th:Tgnew} shows, in particular, that $T_g: A^p_\om\to A^p_\om$ is bounded (resp. compact) if and only if $g\in\B$ (resp. $g\in \B_0$), provided $\om\in\mathcal{B}_\infty(\DDD)$.

We will next study the spectrum of $T_g$ acting on $A^p_\om$, when $\om\in \mathcal{B}_\infty(\DDD)$. We begin with noticing that $T_g$ has no eigenvalues \cite[Proposition~5.1]{AlCo}, and hence its spectrum is nothing else but $\{0\}$ if $g\in\B_0$. The proof
of Theorem~\ref{th:resolvent} follows ideas from the papers \cite{AlCo,AlPe,APR2}, where the approach used reveals a natural connection to weighted norm inequalities for derivatives. This general idea applies to our context as well. Indeed, a simple computation shows that for a  given analytic function $h$ in $\D$ and $\lambda\in\mathbb{C}\setminus\{0\}$, the equation
	$$
	\lambda f-T_gf=h
	$$
has the unique solution $f=\frac{1}{\lambda}R_{\lambda,g}h$, where
	\begin{equation}\label{eq:resol}
	R_{\lambda,g}h(z)=h(0)e^{\frac{g(z)}{\lambda}}+e^{\frac{g(z)}{\lambda}}\int_0^ze^{-\frac{g(\xi)}{\lambda}}h'(\xi)\,d\xi,\quad
	z\in\D.
	\end{equation}
Thus $\lambda$ belongs the resolvent set $\rho\left(T_g|A^p_\om\right)$ if and only if $R_{\lambda,g}$ is a bounded
invertible operator on $A^p_\om$.

\bigskip

\begin{Prf} {\em{Theorem~\ref{th:resolvent}}}. The equivalence between (i) and (ii) follows by arguing as in the proof of \cite[Theorem~5.1]{AlCo} and applying Theorems~\ref{th:LPII} and~\ref{th:Tgnew}.

To see that (ii) and (iii) are equivalent, observe that since $\tom$ is essentially constant in each hyperbolically bounded region, the proof of \cite[Proposition~2.1((b)]{APR2} shows that there exists a differentiable weight $W$ such that $\tom\asymp W$ on $\D$, and 
	\begin{equation}\label{logbloch}
	|\nabla W(z)|\lesssim (1-|z|)W(z),\quad z\in\D.
	\end{equation}  
Therefore, by arguing as in the first part of the proof, but applying Theorem~\ref{th:averageweightII} instead of Theorem~\ref{th:LPII}, we deduce that $\lambda\in\rho\left(T_g|A^p_\om\right)=\rho\left(T_g|A^p_{\widetilde{\om}}\right)=\rho\left(T_g|A^p_{W}\right)$ if and only if  
	\begin{equation*}
	\|f\|_{A^p_{W_{\lambda,g,p}}}^p\asymp |f(0)|^p+\int_\D|f'(z)|^p(1-|z|)^p W_{\lambda,g,p}(z)\,dA(z),\quad f\in\H(\D),
	\end{equation*}
where $W_{\lambda,g,p}=W\exp\left( p \text{Re}\frac{g}{\lambda} \right)$. Since $g\in\B$, also the weight $W_{\lambda,g,p}$ satisfies \eqref{logbloch}. Indeed, 
	\begin{equation*}
	\begin{split}
	\left|\nabla W\exp\left(p\text{Re}\frac{g}{\lambda}\right)(z)\right| 
	&\lesssim\left(|\nabla W(z)|+\frac{p|g'(z)|W(z)}{|\lambda|} \right)\exp\left( p \text{Re}\frac{g}{\lambda}\right)(z)\\
	&\lesssim\frac{W\exp\left( p \text{Re}\frac{g}{\lambda} \right)(z)}{1-|z|},\quad z\in\D. 
	\end{split}
	\end{equation*}
Hence $W_{\lambda,g,p}$ is essentially constant in each hyperbolically bounded region by \cite[Proposition~2.1(i)]{APR2}. Therefore $\lambda \in \rho\left(T_g|A^p_\om\right)=\rho\left(T_g|A^p_{\widetilde{\om}}\right)=\rho\left(T_g|A^p_{W}\right)$ if and only if $W_{\lambda,g,p}\in B_\infty$ by \cite[Corollary~4.4]{APR2}. This is equivalent to $\tom\exp\left( p \text{Re}\frac{g}{\lambda} \right)\in B_\infty$. 
\end{Prf}

\end{document}